\documentclass{cibb}

\usepackage{subcaption} 
\usepackage{float} 
\usepackage{multicol}
\usepackage{caption} 
\usepackage{dirtytalk}
\usepackage{subcaption}
\usepackage{graphicx}
\usepackage{amsmath,amsfonts,latexsym,amssymb,euscript,xr}
\usepackage{booktabs}
\usepackage[nodayofweek]{datetime}
\usepackage[table]{xcolor}
\usepackage{hyperref}
\usepackage{fmtcount}
\usepackage[english]{datenumber}
\usepackage{csquotes}
\usepackage[absolute]{textpos}
 \usepackage{amsthm} 
 \usepackage{amsmath}
 \usepackage{amssymb} 

\newtheorem{theorem}{Theorem}[section]
\newtheorem{corollary}{Corollary}[theorem]

\newtheorem{proposition}{Proposition}[section]
\theoremstyle{definition}
\newtheorem{definition}{Definition}[section]
\newtheorem{conjecture}{Conjecture}

\newcommand{\R}[0]{\mathbb{R}}

\newcommand{\vol}[2]{\text{vol}_{#1}\left( #2 \right)}
\usepackage{pgfplots}
\pgfplotsset{compat=1.15}
\usepackage{mathrsfs}
\usetikzlibrary{arrows}

\usepackage{color,colortbl,tabularx}

\usepackage[english]{babel}
\usepackage[protrusion=true,expansion=true]{microtype}
\usepackage{amsmath,amsfonts,amsthm}
\usepackage{pifont}

\definecolor{LightBlue}{rgb}{0.88,0.9,0.9}

\title{\Large $\ $\\ \bf Hedgehog Reconstruction of Polygons: \\Non-Central Sections and Slabs}

\author{\large Brendan Matthews\footnote{The author was supported by an NSERC USRA 2025. 
 }}
\address{Department of Mathematics and Statistics, University of the Fraser Valley,\\
Abbotsford, Canada V2S 7M7 \\
E-mail: BrendanMatthews.academia@gmail.com
}

\abstract{ We show that a polygon can be uniquely determined by the lengths of non-central sections supporting a piecewise-analytic hedgehog in the interior of the polygon. We also prove the analogous result for slab areas - an origin-symmetric polygon can be reconstructed based on the areas of slabs supporting an analytic origin-symmetric hedgehog in the interior of the polygon.}

\begin{document}


\thispagestyle{myheadings}
\pagestyle{myheadings}

\section{Introduction}
\label{sec:Introduction}

Problems concerning the reconstruction of convex bodies from data on their sections or projections arise naturally at the intersection of classical mechanics, harmonic analysis, differential geometry, and convex geometry \cite{AlfonsecaStancuRyaboginYaskin, BusemannPetty1956, Werner2006, Gardner_2006}. In particular, the questions addressed in this work trace their origins to Hammer's X-ray problem \cite{Falconer1983PointSources,Falconer1983StableManifold,GardnerMcMullen1980,Hammer1963,Volcic1986}.
In \cite{BarkerLarman2001},  Barker and Larman posed the following,

\begin{conjecture} \label{conj1}
    If for two convex bodies $K$ and $L$ in $\R^n$, $n \geq 2$, with a convex body $M$ in their interiors and a supporting plane $H_M(\xi)$ of $M$ with unit normal vector $\xi$, the $(n-1)$-dim volumes of the sections coincide,
\[ \text{vol}_{n-1}\left( K \cap H_M(\xi) \right) = \text{vol}_{n-1}(L \cap H_M(\xi)) \quad \text{for all} \quad \xi \in S^{n-1}, 
\]
then $K=L$.
\end{conjecture}

In \cite{Xiong2008Determination}, Xiong, Ma and Cheung confirmed the conjecture for origin-symmetric, strictly convex body $M$ with analytic boundary in the planar case with $K$ and $L$ being polygons. Later, in \cite{Yaskin2011}, Yaskin gave the affirmative answer for convex polytopes $K,L$ and a ball $M$ (so-called $t$-sections). One of the goals of this paper is to generalize this to a larger class of interior bodies $M$ for $n=2$ that are not necessarily convex. In particular, we are interested in envelopes of lines defined by an analytic support function, \textit{analytic hedgehogs}~\cite{MartinezMaure2023}. 


\begin{theorem}\label{Thm1}
    Let $P$ and $Q$ be convex polygons in $\R^2$ with a non-trivial piecewise analytic hedgehog $M$ in their interiors. If for every supporting line $H$ of $M$,
    \[ \vol{1}{P \cap H}= \vol{1}{Q \cap H}, \]
    then $P= Q$.
\end{theorem}
To formulate the second result, we recall the notion of \textit{slab}, which is the set of points between two parallel hyperplanes. If parallel planes with unit normal $\xi \in S^{n-1}$ are chosen to be tangent to a ball of radius $t$, we denote the slab $S_t(\xi)$. In \cite{RyaboginYaskin2013}, Ryabogin and Yaskin stated the following,
\begin{conjecture} 
    If for two origin-symmetric convex bodies $K$ and $L$ in $\R^n$, $n \geq 2$, with a ball in their interiors,
    \[ \text{vol}_{n}\left( K \cap S_t(\xi) \right) = \text{vol}_n(L \cap S_t(\xi)) \quad \text{for all} \quad \xi \in S^{n-1},\]
then $K=L$. \label{conj2}
\end{conjecture}
Simultaneously, the authors emphasized that the condition of origin-symmetry cannot be eliminated (Figure \ref{fig:counter}). Consider a rectangle centred at the origin. Cut one corner of the rectangle to obtain a polygon $P$. To construct a polygon $Q$, rotate $P$ by $180$ degrees about the origin. Then we obtain two distinct non-origin-symmetric polygons with the same slab areas in any direction. 

In \cite{Yaskin2015}, Yaskin and Yaskina provided the affirmative answer to Conjecture \ref{conj2} in the class of convex polytopes. The next theorem extends this observation in the plane to the class of analytic hedgehogs instead of a ball.
\begin{theorem}\label{Thm2}
 Let $P$ and $Q$ be origin-symmetric convex polygons in $\R^2$ with a non-trivial analytic origin-symmetric hedgehog $M$ in their interiors. 
 If
\[ 
\vol{2}{P \cap S_M(\xi)} = \vol{2}{Q \cap S_M(\xi)} \quad \text{for all} \quad \xi \in S^1,
\]
    then $P =Q$.
\end{theorem}
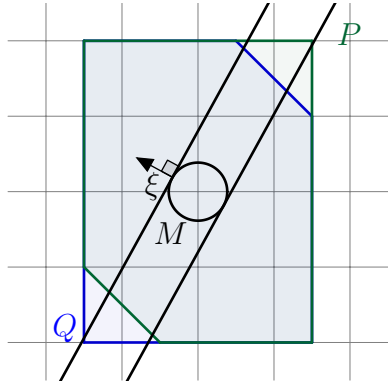
\begin{figure}[h]
    \centering
    \caption{$P \neq Q$, but $\vol{2}{P \cap S_M(\xi)} = \vol{2}{Q \cap S_M(\xi)}$ for all $\xi \in S^1$.}
   \definecolor{qqqqff}{rgb}{0.,0.,1.}
\definecolor{qqwwtt}{rgb}{0.,0.4,0.2}
\definecolor{qqqqcc}{rgb}{0.,0.,0.8}
\begin{tikzpicture}[line cap=round,line join=round,>=triangle 45,x=0.5cm,y=0.5cm]
\clip(-5.,-5.) rectangle (5.,5.);
\draw[step=1cm,gray,very thin] (-5,-5) grid (5,5);
\fill[line width=1.pt,color=qqqqcc,fill=qqqqcc,fill opacity=0.05000000074505806] (1.,4.) -- (3.,2.) -- (3.,-4.) -- (-3.,-4.) -- (-3.,4.) -- cycle;
\fill[line width=1.pt,color=qqwwtt,fill=qqwwtt,fill opacity=0.05000000074505806] (-3.,-2.) -- (-1.,-4.) -- (3.,-4.) -- (3.,4.) -- (-3.,4.) -- cycle;
\draw[line width=.5pt,fill=black,fill opacity=0.10000000149011612] (-0.5087348013620465,0.6748560175577366) -- (-0.8024029209017183,0.8352876424500033) -- (-0.9628345457939849,0.5416195229103316) -- (-0.6691664262543131,0.3811878980180649) -- cycle; 
\draw [line width=1.pt,color=qqqqcc] (1.,4.)-- (3.,2.);
\draw [line width=1.pt,color=qqqqcc] (3.,2.)-- (3.,-4.);
\draw [line width=1.pt,color=qqqqcc] (3.,-4.)-- (-3.,-4.);
\draw [line width=1.pt,color=qqqqcc] (-3.,-4.)-- (-3.,4.);
\draw [line width=1.pt,color=qqqqcc] (-3.,4.)-- (1.,4.);
\draw [line width=1.pt,color=qqwwtt] (-3.,-2.)-- (-1.,-4.);
\draw [line width=1.pt,color=qqwwtt] (-1.,-4.)-- (3.,-4.);
\draw [line width=1.pt,color=qqwwtt] (3.,-4.)-- (3.,4.);
\draw [line width=1.pt,color=qqwwtt] (3.,4.)-- (-3.,4.);
\draw [line width=1.pt,color=qqwwtt] (-3.,4.)-- (-3.,-2.);
\draw[line width=1.pt, smooth,samples=100,domain=0.0:6.283185307179586] plot[parametric] function{0.77*cos((t))-0.0*sin((t)),0.77*sin((t))+0.0*cos((t))};
\draw[line width=1.pt] (-8.397216539766491,6.893259497090827) -- (-5.242265993233279,6.893259497090827);
\draw [line width=1.pt,domain=-5.:5.] plot(\x,{(--0.77-0.8775825618903728*\x)/-0.479425538604203});
\draw [line width=1.pt,domain=-5.:5.] plot(\x,{(-0.77-0.8775825618903728*\x)/-0.479425538604203});
\draw [color=qqwwtt](3.354974246069723,4.76366787818091) node[anchor=north west] {$P$};
\draw [color=qqqqff](-4.122258549213989,-2.950186208092786) node[anchor=north west] {$Q$};
\draw (-1.4568847335521148,-0.5155228036582112) node[anchor=north west] {$M$};
\draw [->,line width=.5pt] (-0.6691664262543131,0.3811878980180649) -- (-1.6456223701854162,0.9146282114103952);
\draw (-1.6929466283834158,0.8199796950143989) node[anchor=north west] {$\xi$};
\end{tikzpicture}
    \label{fig:counter}
\end{figure}


It is natural to wonder whether the analogous result holds if the equality for slab areas in Theorem \ref{Thm2} is replaced by the equality for their perimeters. Unfortunately, the series expansion approach helpful for the proof of Theorem \ref{Thm2} in this case does not seem to yield a clear conclusion, and the answer to such a question of uniqueness eludes us.

The proofs of both theorems are carried out by contradiction. We assume that the tomographic data (lengths of sections or areas of slabs) coincide for a pair of distinct polygons. These assumptions guarantee the existence of a point where the boundaries of the polygons intersect transversely. We then consider a supporting line of the hedgehog passing through this point, and equate the corresponding tomographic data on an open set of nearby supporting lines. The assumption of analyticity allows us to uniquely extend this equality to the entire unit circle of normal directions, except for a finite set of points. This ultimately contradicts the transversality of the intersection and therefore implies the equality of the polygons.
These arguments rely heavily on the polygonal structure of the bodies, since it is essential to vary the supporting lines over an open set while intersecting the same sides of the polygons.

These results also raise further questions concerning stability, extensions beyond the polyhedral class, higher dimensions, and other intrinsic volumes \cite{Gardner_2006}.


\section{Preliminaries}
\label{sec:DATA-AND-METHODS}

For the benefit of a reader, we recall the necessary preliminary notions and results. 
\textbf{\subsection{Support function and envelopes}}

A set $K$ in $\mathbb{R}^n$ is called \textit{convex} if for any two points in $K$ the line segment connecting them is contained in $K$. Set $K$ is called \textit{a convex body} if it is a bounded closed convex set with non-empty interior (Figure \ref{fig:convex}).

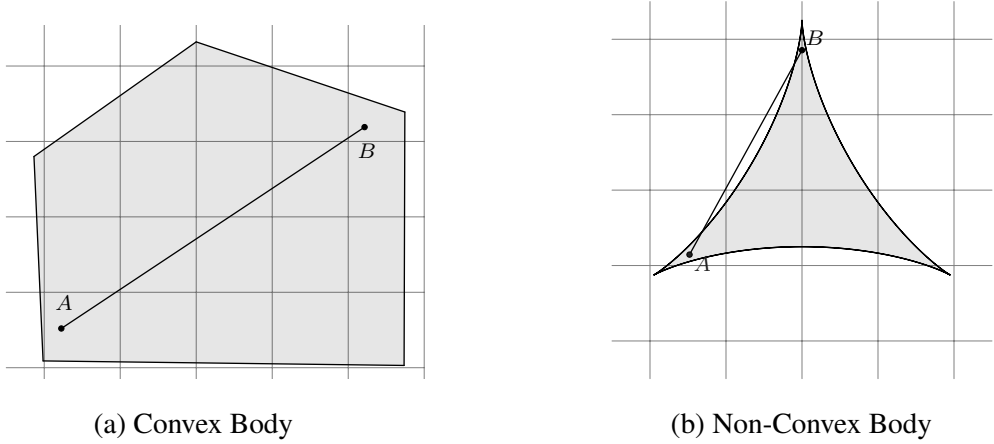
\begin{figure}
\centering
\caption{Convex and Non-Convex Bodies in $\R^2$.}
\begin{subfigure}{.5\textwidth}
 \centering 
\begin{tikzpicture}[line cap=round,line join=round,>=triangle 45,x=0.5cm,y=0.5cm]
\clip(-8.05866666666667,-4.297777777777782) rectangle (8.553777777777785,5.069333333333336);
\draw[step=1cm,gray,very thin] (-5,-6) grid (6,6);
\fill[line width=2.pt,fill=black,fill opacity=0.1] (0.,4.64) -- (-4.27,1.6) -- (-4.03,-3.82) -- (5.47,-3.94) -- (5.49,2.78) -- cycle;
\draw [line width=.5pt] (0.,4.64)-- (-4.27,1.6);
\draw [line width=.5pt] (-4.27,1.6)-- (-4.03,-3.82);
\draw [line width=.5pt] (-4.03,-3.82)-- (5.47,-3.94);
\draw [line width=.5pt] (5.47,-3.94)-- (5.49,2.78);
\draw [line width=.5pt] (5.49,2.78)-- (0.,4.64);
\draw [line width=.5pt] (-3.55,-2.96)-- (4.43,2.38);
\begin{scriptsize}
\draw [fill=black] (-3.55,-2.96) circle (1.pt);
\draw[color=black] (-3.4826666666666677,-2.277333333333336) node {$A$};
\draw [fill=black] (4.43,2.38) circle (1.pt);
\draw[color=black] (4.496,1.7613333333333346) node {$B$};
\end{scriptsize}
\end{tikzpicture}
\subcaption{Convex Body}  
\label{fig:convex}
\end{subfigure}%
\begin{subfigure}{.5\textwidth}
  \centering
         \definecolor{sqsqsq}{rgb}{0.12549019607843137,0.12549019607843137,0.12549019607843137}
\begin{tikzpicture}[line cap=round,line join=round,>=triangle 45,x=0.5cm,y=0.5000000002673796cm]
\draw[step=1cm,gray,very thin] (-5,-5) grid (5,5);
\clip(-5.,-5.) rectangle (5.,5.);

\draw[line width=.5pt, smooth, samples=300, variable=\a, domain=0:360]
  plot ({(3/2)*sin(3*\a)*cos(\a) - (9/2)*cos(3*\a)*sin(\a)},
        {(3/2)*sin(3*\a)*sin(\a) + (9/2)*cos(3*\a)*cos(\a)});

\path[fill=black, opacity=0.1]
  plot[smooth, samples=300, variable=\a, domain=0:360]
    ({(3/2)*sin(3*\a)*cos(\a) - (9/2)*cos(3*\a)*sin(\a)},
     {(3/2)*sin(3*\a)*sin(\a) + (9/2)*cos(3*\a)*cos(\a)})
  -- cycle;

\draw [line width=.5pt] (0.,3.71)-- (-2.96,-1.71);

\begin{scriptsize}
\draw [fill=sqsqsq] (-2.96,-1.71) circle (1.0pt);
\draw[color=sqsqsq] (-2.62,-1.94) node {$A$};
\draw [fill=black] (0.,3.71) circle (1.0pt);
\draw[color=black] (0.34,4.04) node {$B$};
\end{scriptsize}
\end{tikzpicture}

  \subcaption{Non-Convex Body}
  \label{fig:non-convex}
\end{subfigure}

\label{fig:convex-non-convex}
\end{figure}

Among those, \textit{the origin-symmetric} convex bodies are symmetric about the origin, i.e., for a point $v=(v_1,v_2) \in K$, we have $-v=(-v_1,-v_2) \in K$ as well.
For two vectors $u=(u_1, u_2), v=(v_1,v_2)$ in $\R^2$, the standard \textit{inner product} is denoted as $\langle u,v \rangle = u_1v_1 + u_2v_2$. The \textit{unit circle} $S^1$ is the set of unit vectors in $\R^2$.

For a convex body $K$ in $\R^2$, \textit{a supporting line} $H$ of $K$ is a line that intersects the boundary of $K$ but not the interior of $K$. 
\textit{The support function} $h_K$ of $K$ corresponds $\xi \in S^1$ to the signed distance from the origin to the supporting line $H(\xi)$ with the normal vector $\xi$.
Thus,
\begin{equation*}
h_K(\xi) = \sup\left\{ \langle \xi ,y \rangle \ | \ y \in K \right\}, \label{def:supportfuntion}
\end{equation*} 
and
\begin{equation*} 
H(\xi) = \{ x \in \R^2 \ | \  \langle x , \xi \rangle = h_K(\xi) \}. \label{def:supportline}
\end{equation*}
For $\xi \in S^1$, two parallel supporting lines of $K$ with normal vector $\xi$ determine a slab $S_K(\xi)$,
\begin{equation*} 
S_K(\xi) = \left\{ x \in \R^2| -h_K(-\xi) \leq \langle x , \xi \rangle \leq h_K(\xi) \right\}. \label{def:slab}
\end{equation*} 

As Figure \ref{fig:slab-int-polygon} indicates, we are interested in the intersection of a convex body $P$ with a slab defined by $E$.
\definecolor{qqwuqq}{rgb}{0.,0.39215686274509803,0.}
\definecolor{qqwwtt}{rgb}{0.,0.4,0.2}
\definecolor{ududff}{rgb}{0.30196078431372547,0.30196078431372547,1.}
\begin{figure}[H]
    \centering
    \caption{Intersection of a slab $S_E(\xi)$ with $P$.}
    \begin{tikzpicture}[
        scale=0.22,
        line cap=round,
        line join=round,
        >=triangle 45,
        x=1.0cm,
        y=1.0cm
    ]

    \clip(-15,-10) rectangle (15,10);
    \draw[step=3cm,gray,very thin] (-15,-10) grid (15,10);

    \draw[line width=1pt]
    (-3.7932417995731744,7.656946484500107) --
    (-11.168060539280887,-1.4591489020830366) --
    (-1.130112810234278,-8.76568228308975) --
    (10.,2.) --
    (0.,8.) -- cycle;

    \filldraw[
        line width=1pt,
        draw=ududff,
        fill=ududff,
        fill opacity=0.25
    ]
    (-8.814355110159374,1.4502925311366117) --
    (-5.956109809658631,4.983401305366698) --
    (7.172345866614307,-0.7350689544228514) --
    (3.7663169723491663,-4.0295750881760615) -- cycle;

    \filldraw[
        line width=1pt,
        draw=qqwuqq,
        fill=qqwuqq,
        fill opacity=0.25
    ]
    (0,0) ellipse [x radius=3, y radius=2];

    \draw[line width=1pt]
    (1.4692346436474342,1.749082473665502) --
    (1.728214761274243,2.343648649729629) --
    (1.1336485852101155,2.602628767356438) --
    (0.8746684675833066,2.008062591292311) -- cycle;

    \draw[->,line width=1pt]
    (0.8746684675833066,2.008062591292311) --
    (1.9500000000000000,4.4760000000000000);

    \draw[line width=1pt,domain=-15:15]
    plot(\x,{(2.190287720536304 - 0.3993395294062732*\x)/0.9168031087717669});

    \draw[line width=1pt,domain=-15:15]
    plot(\x,{(-2.190287720536304 - 0.3993395294062732*\x)/0.9168031087717669});

    \draw (2.5989210994372103,8.709716371264202)
    node[anchor=north west] {$P$};

    \draw (0,0)
    node {$E$};

    \draw (-4.8,2.3)
    node {$S_E(\xi)$};

    \draw (2.1123453190861666,4.7354114218966)
    node[anchor=north west] {$\xi$};

    \end{tikzpicture}
    \label{fig:slab-int-polygon}
\end{figure}
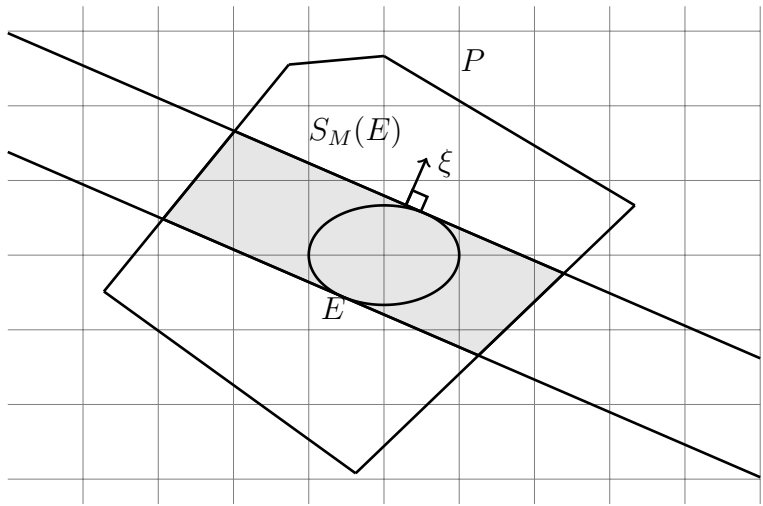

To extend the notion of slab $S_K$ for a convex body $K$ to non-convex structures, we recall the notion of envelope \cite{BruceGiblin1992}. For a family of lines in the plane, the envelope (if it exists) can be thought of as a curve on which each point belongs to at least one of the lines.
\begin{definition}
For $\theta \in \R$ and a differentiable $F$, let a family of lines be defined by $F(\theta,x_1,x_2)=~0$. The \textit{envelope} is the set of points $(x_1,x_2) \in \mathbb{R}^2$ that satisfy
\[ F(\theta,x_1,x_2) = 0, \quad \quad \frac{\partial F}{\partial \theta}(\theta,x_1,x_2)=0 \quad \text{for some} \quad \theta \in \mathbb{R}.\]
\end{definition}
For example, the envelope of the family of supporting lines for a convex body $K$ is $\partial K$, the boundary of $K$.
If the family of lines is defined by a smooth function, each line has a unique common point $X_\xi$ with the envelope, \textit{supporting point}. The following well-known proposition provides a way to trace $X_\xi$.
\begin{proposition}({\cite{MartinezMaure2023}, pp. 15--16})
    Given an envelope $M \subset \R^2$ with a smooth support function $h_M$ and $\xi = (\cos \theta, \sin \theta) \in S^1$, the unique support point $X_\xi$ is given by
\begin{equation}
    \begin{split}
         X_\xi & = 
         h_M(\theta) \left( \cos \theta, \sin \theta \right) + \frac{dh_M(\xi)}{d \theta} \left( -\sin \theta, \cos \theta \right).
    \label{def:supporting_point}
    \end{split}
\end{equation}
\end{proposition}

\textbf{\subsection{Hedgehogs}}
An envelope of lines that does not necessarily result in a convex body can have a peculiar shape that is reflected in its name, Figures \ref{fig:Hedgehog1} and \ref{fig:Hedgehog2}.
\begin{definition}
    For $\xi \in S^1$ and a smooth function $h_M(\xi)$ on $S^1$, a \textit{hedgehog} $M$ is the envelope of the family of lines defined by $\langle (x_1,x_2) ,\xi \rangle = h_M(\xi)$.
\end{definition}

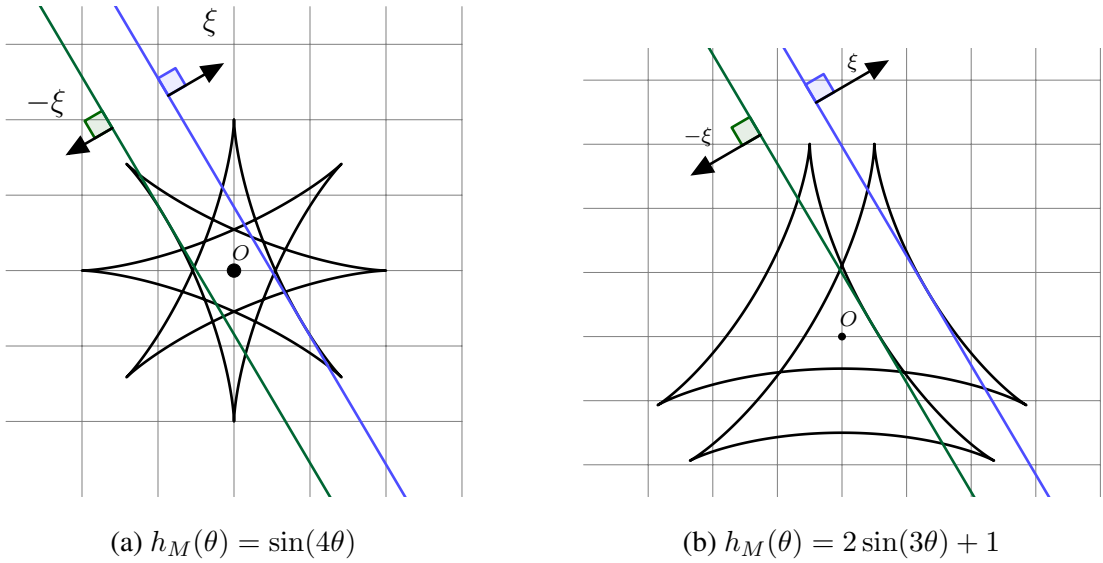
\begin{figure}[H]
\centering
\caption{Analytic hedgehogs.}

\begin{subfigure}{.5\textwidth}
 \centering 
\definecolor{qqwuqq}{rgb}{0.,0.39215686274509803,0.}
\definecolor{qqwwtt}{rgb}{0.,0.4,0.2}
\definecolor{ududff}{rgb}{0.30196078431372547,0.30196078431372547,1.}
\begin{tikzpicture}[scale=1.0, line cap=round, line join=round, >=triangle 45, x=0.5cm, y=0.5cm]
\clip(-6.,-6.) rectangle (6.,7.);
\draw[step=1cm,gray,very thin] (-6,-6) grid (6,7);

\draw[line width=1.pt,color=qqwuqq,fill=qqwuqq,fill opacity=0.10000000149011612]
(-3.4706661189950756,4.236268136554644) -- (-3.928649294903929,3.9679280023057686) --
(-3.6603091606550535,3.5099448263969153) -- (-3.2023259847462002,3.7782849606457907) -- cycle; 

\draw[line width=1.pt,color=ududff,fill=ududff,fill opacity=0.10000000149011612]
(-1.2724966524131733,4.909004797871259) -- (-1.5408367866620487,5.3669879737801125) --
(-1.998819962570902,5.098647839531237) -- (-1.7304798283220268,4.640664663622384) -- cycle; 

\draw[line width=1.pt, smooth, samples=200, variable=\t, domain=0:6.283185307179586]
  plot ({sin(4*\t r)*cos(\t r) - 4*cos(4*\t r)*sin(\t r)},
        {sin(4*\t r)*sin(\t r) + 4*cos(4*\t r)*cos(\t r)});

\draw[line width=4.pt] (-13.22897689714715,10.335297679591154) -- (-8.22448986991725,10.335297679591154);
\draw [line width=1.pt,color=ududff,domain=-6.:6.] plot(\x,{(--0.8529404815528762-0.862807070514761*\x)/0.5055333412048469});
\draw [line width=1.pt,color=qqwwtt,domain=-6.:6.] plot(\x,{(-0.8529404815528759-0.862807070514761*\x)/0.5055333412048469});
\draw (-1.1181182912507872,7.257538369179147) node[anchor=north west] {$\xi$};
\draw [->,line width=1.pt] (-3.2023259847462002,3.7782849606457907) -- (-4.436254054286156,3.0553055123795416);
\draw [->,line width=1.pt] (-1.7304798283220268,4.640664663622384) -- (-0.2586724139586658,5.503021666966443);
\draw (-5.7722912265745965,5.1056090952325395) node[anchor=north west] {$-\xi$};

\begin{scriptsize}
\draw [fill=black] (-12.806837807695011,10.335297679591154) circle (2.5pt);
\draw[color=black] (-12.478303843062665,10.923324864913774) node {$u = 0.53$};
\draw[color=qqwwtt] (-7.59892899151351,11.874177334797158) node {$eq2$};
\draw [fill=black] (0.,0.) circle (2.5pt);
\draw[color=black] (0.18304833582898727,0.46394769619654486) node {$O$};
\end{scriptsize}
\end{tikzpicture}

\subcaption{$h_M(\xi) = \sin(4\theta)$}  
\label{fig:Hedgehog1}
\end{subfigure}%
\begin{subfigure}{.5\textwidth}
  \centering
\definecolor{qqwuqq}{rgb}{0.,0.39215686274509803,0.}
\definecolor{qqwwtt}{rgb}{0.,0.4,0.2}
\definecolor{ududff}{rgb}{0.30196078431372547,0.30196078431372547,1.}
\begin{tikzpicture}[scale=.85, line cap=round, line join=round, >=triangle 45, x=0.5cm, y=0.5cm]
\clip(-8.,-5.) rectangle (8.,9.);
\draw[step=1cm,gray,very thin] (-8,-5) grid (8,9);

\draw[line width=1.pt,color=qqwuqq,fill=qqwuqq,fill opacity=0.10000000149011612]
(-2.854717591082678,6.849601143862958) -- (-3.4142661961090175,6.521752096802283) --
(-3.0864171490483425,5.962203491775943) -- (-2.5268685440220033,6.290052538836618) -- cycle; 

\draw[line width=1.pt,color=ududff,fill=ududff,fill opacity=0.10000000149011612]
(-0.241705797966142,7.628968268306987) -- (-0.5695548450268167,8.188516873333326) --
(-1.1291034500531558,7.8606678262726515) -- (-0.8012544029924813,7.301119221246312) -- cycle; 

\draw[line width=1.pt, smooth, samples=250, variable=\t, domain=0:6.283185307179586]
  plot ({(2*sin(3*\t r)+1)*cos(\t r) - 6*cos(3*\t r)*sin(\t r)},
        {(2*sin(3*\t r)+1)*sin(\t r) + 6*cos(3*\t r)*cos(\t r)});

\draw[line width=1.pt] (-13.818016626373494,11.82801717769109) -- (-7.703700899814571,11.82801717769109);
\draw [line width=1.pt,color=ududff,domain=-8.:8.] plot(\x,{(--2.999631230268582-0.862807070514761*\x)/0.5055333412048469});
\draw [line width=1.pt,color=qqwwtt,domain=-8.:8.] plot(\x,{(--0.9996312302685817-0.862807070514761*\x)/0.5055333412048469});
\draw [->,line width=1.pt] (-2.5268685440220033,6.290052538836618) -- (-4.695101239969814,5.0196481150085495);
\draw [->,line width=1.pt] (-0.8012544029924813,7.301119221246312) -- (1.4583752449927445,8.625074674172751);

\begin{scriptsize}
\draw [fill=black] (-13.30226113347512,11.82801717769109) circle (1.5pt);
\draw[color=black] (-12.900869267389655,12.546449226230619) node {$u = 0.53$};
\draw[color=qqwwtt] (-6.4502661758699915,13.708169134507301) node {$eq2$};
\draw [fill=black] (0.,0.) circle (1.5pt);
\draw[color=black] (0.21433796607923447,0.5623912250606176) node {$O$};
\draw[color=black] (-4.4542514698561195,6.220702388839291) node {$-\xi$};
\draw[color=black] (0.33662428061041294,8.558143241127838) node {$\xi$};
\end{scriptsize}
\end{tikzpicture}

\subcaption{$h_M(\xi) = 2\sin(3\theta) +1$}
\label{fig:Hedgehog2}
\end{subfigure}

\label{fig:hedgehogs}
\end{figure}

In particular, the definition implies that $h_M = h_M(\xi(\theta))$ is $2\pi$-periodic. Moreover, \eqref{def:supporting_point} does not depend on the convexity assumption, and can be used to parameterize $M$ for $\theta \in [0, 2\pi]$.
This way the notion of support line $H(\xi)$ for a hedgehog $M$ extends to $\langle (x_1,x_2) ,\xi \rangle = h_M(\xi)$.  
See \cite{MartinezMaure2023} for more information on hedgehog theory and its fruitful applications.

\textbf{\subsection{Real Analytic Functions}}
\begin{definition}[\cite{KrantzParks2002}, p. 3] \label{def:analytic}
    Let $U$ be an open set in $\mathbb{R}$. Function $f$ is said to be \textit{real analytic} at a point $c \in U$ if $f$ is equal to a convergent power series in some neighbourhood of $c$, 
    \[ f(t) = \sum_{k=0}^\infty a_k(t-c)^k, \quad \text{for all} \quad t \in (c-\varepsilon, c+\varepsilon) \quad \text{for some} \quad \varepsilon > 0.
    \]
\end{definition}

A function $f$ is said to be real analytic on a open set $U$ if $f$ is real analytic at every point of $U$. 
Recall that the basic arithmetic operations on two analytic functions yield an analytic function,
\begin{theorem}[\cite{KrantzParks2002}, pp. 4--6]
Let $f$ be real analytic function on $A$ and $g$ be real analytic function on $B$, then the following holds
     \begin{itemize}
        \item $f \pm g$ is real analytic on $A \cap B$
        \item $f\cdot g$ is real analytic on $A \cap B$
        \item $f/g$ is real analytic on $(A \cap B )/\{ t\in \R| \, g(t)=0\}$
    \end{itemize} \label{thrm:analytic-combination}
\end{theorem} 

A profound property of unique analytic extension plays a vital role in the following proofs. 
\begin{theorem}[\cite{KrantzParks2002}, p. 13]
    If $f$ and $g$ are real analytic on an open set $U$ and there is an open set $W \subset U $ such that,
    \[ f(t) = g(t), \quad \text{for all} \quad t  \in W,\]
    then
    \[ f(t) = g(t) \quad \text{for all} \quad t  \in U.\] \label{thrm:eqnext}
\end{theorem}
We also deal with \textit{piecewise analytic functions}, i.e. functions that are not analytic on a finite set of points but analytic otherwise. 


\section{Proofs of the Theorems}
\label{sec:RESULTS}

In this section, we provide the proofs of Theorem \ref{Thm1} and Theorem \ref{Thm2} in the analytic settings that allow to generalize and simplify some of the approaches taken in \cite{Yaskin2011,Xiong2008Determination, Yaskin2015}.
\textbf{\subsection{Proof of Theorem \ref{Thm1}} }
 \begin{proof}
    Suppose to the contrary that $P$ and $Q$ are two distinct convex polygons in $\R^2$ that contain an analytic hedgehog $M$ in their respective interiors, and for every supporting line $H$ of $M$,  $\vol{1}{P \cap H} = \vol{1}{Q \cap H}$. Hence, neither of the polytopes is contained in the interior of the other, otherwise $\vol{1}{P \cap H} \neq \vol{1}{Q \cap H}$. Thus, there must be a point of intersection of the boundaries of $P$ and $Q$ where they intersect transversally. Let this point of intersection be $u_1 \in \partial P \cap \partial Q$. Then, we consider a parametrization of the corresponding edges of $P$ and $Q$ containing $u_1$,    
    \begin{equation}\label{eqn:edges1}
         x(s_1) = u_1 + \ell_1 s_1, \qquad x(t_1) = u_1 + m_1t_1, \qquad t_1,s_1 \in \mathbb{R}, \quad \ell_1,m_1 \in S^1.
     \end{equation} 
Since the boundaries intersect transversely at $u_1$, we must have $\ell_1$ is not parallel to $m_1$. 

Let $\xi_0 \in  S^1$ such that $H(\xi_0)$ passes through $u_1$ and is a supporting line of $M$. Then, $H(\xi_0)$ must also contain another intersection point of the boundaries of $P$ and $Q$, otherwise $\vol{1}{P \cap H} \neq \vol{1}{Q \cap H} $. Similarly, the parametrization of the other pair of the intersected edges is
\begin{equation}\label{eqn:edges2}
    x(s_2) = u_2+ \ell_2 s_2, \quad x(t_2) = u_2 + m_2t_2, \qquad t_2,s_2 \in \mathbb{R}, \quad \ell_2,m_2 \in S^1,
\end{equation} 
where $x(s_1) = u_1 + \ell_1 s_1$ and $x(s_2) = u_2 + \ell_2 s_2$ are edges of $P$, and $x(t_1) = u_1 + m_1t_1$ and $x(t_2) = u_2 + m_2t_2$ are edges of $Q$. 
Consider a small open interval $\Lambda$ of $S^1$ such that
$$
\Lambda = \left\{  \xi \in S^1 | H(\xi) \text{ intersects edges (\ref{eqn:edges1}) and (\ref{eqn:edges2})}, \, H(\xi) \, \text{contains no vertices of} \, P \, \text{or} \, Q \right\}. 
$$
For any $\xi \in \Lambda$, there are 4 points of intersection with the sides of the polygons, $p_1,\  p_2,\ q_1,\ q_2$. To determine these points, we solve the following systems of equations,
\[\begin{cases}
    \langle x , \xi \rangle = h_M(\xi) \\
    x = u_i + \ell_i s_i 
\end{cases} \qquad \begin{cases}
    \langle x , \xi \rangle = h_M(\xi) \\
    x = u_i + m_i t_i.
\end{cases}  \]    \\
For $i =1, 2$, these yield
$$
    p_i = u_i + \ell_i \left( \frac{h_M(\xi) - \langle u_i ,\xi\rangle }{\langle \ell_i ,\xi \rangle }\right), \quad
    q_i = u_i + m_i \left( \frac{h_M(\xi) - \langle u_i ,\xi\rangle }{\langle m_i ,\xi \rangle }\right).
$$
The condition $\vol{1}{P \cap H} = \vol{1}{Q \cap H}$ implies the equality between vectors $p_1-p_2$ and $q_1-q_2$ for $\xi \in \Lambda$,
\begin{equation}
\begin{split}
   & \ell_1 \left( \dfrac{h_M(\xi)-\langle u_1,\xi \rangle}{ \langle \ell_1 , \xi \rangle} \right) - \ell_2 \left( \dfrac{h_M(\xi)-\langle u_2,\xi \rangle}{ \langle \ell_2 , \xi \rangle} \right)  \\
    & = m_1 \left( \dfrac{h_M(\xi)-\langle u_1,\xi \rangle}{ \langle m_1 , \xi \rangle} \right) -  m_2 \left( \dfrac{h_M(\xi)-\langle u_2,\xi \rangle}{ \langle m_2 , \xi \rangle} \right).
    \label{eqn:sections-equality} \\
\end{split}
\end{equation}
Since $h_M(\xi)$ is piecewise analytic on $S^1$, there exists a finite set $X = \{ x_1 , x_2, \ldots , x_n, \} \subset S^1$, where $h_M$ is not analytic and $h_M$ is analytic on every subinterval of $S^1/X$. 
A vector $ \xi \in S^1$ is parameterized in terms of angle $\theta \in [-\pi, \pi]$, $\xi = (\cos \theta, \sin \theta)$. 
Let both sides of equation \ref{eqn:sections-equality} be,
\begin{align*} 
&f(\xi) = \ell_1 \left( \dfrac{h_M(\xi)-\langle u_1,\xi \rangle}{ \langle \ell_1 , \xi \rangle} \right) - \ell_2 \left( \dfrac{h_M(\xi)-\langle u_2,\xi \rangle}{ \langle \ell_2 , \xi \rangle} \right),\\
 &g(\xi) = m_1 \left( \dfrac{h_M(\xi)-\langle u_1,\xi \rangle}{ \langle m_1 , \xi \rangle} \right) -  m_2 \left( \dfrac{h_M(\xi)-\langle u_2,\xi \rangle}{ \langle m_2 , \xi \rangle} \right).
\end{align*}
By Theorem \ref{thrm:analytic-combination}, $f(\xi)$ is analytic on $S^1 /\left(X \cup \left\{ \ell_1^\perp, \ell_2^\perp \right\}\right )$. Analogously, $g(\xi)$ is analytic on $S^1 /\left(X \cup \left\{ m_1^\perp, m_2^\perp\right\}\right)$. Now consider the $4$ lines that contain the intersected sides of the polytopes, and extend the measurements of the corresponding segments between the lines analytically (and, by Theorem \ref{thrm:eqnext}, uniquely), beyond the sides of polygons $P$ and $Q$. 


Since $f(\xi) = g(\xi)$ for $\xi \in \Lambda$, then the set of directions on $S^1$ where $f(\xi)$ is not analytic coincides with the set of such directions for $ g(\xi)$, $\{X,m_1^\perp, m_2^\perp\}=\{X,\ell_1^\perp, \ell_2^\perp\}$, that gives $\{\ell_1^\perp, \ell_2^\perp\} = \{m_1^\perp, m_2^\perp\}$.

\subsubsection{Case \texorpdfstring{$\ell_1^\perp = m_1^\perp$}{l1 perp = m1 perp} (\texorpdfstring{$\ell_1$}{l1} is parallel to \texorpdfstring{$m_1$}{m1})}
This is a contradiction to the assumption on transversality of the edges assumed in $\eqref{eqn:edges1}$.



\subsubsection{Case \texorpdfstring{$\ell_1^\perp = m_2^\perp$ ($\ell_1$ is parallel to $m_2$)}{l1 perp = m2 perp (l1 is parallel to m2)}}
In this case, $\ell_1 = \pm m_2$ and $\ell_2 = \pm m_1$. Then $(\ref{eqn:sections-equality})$ implies,
\begin{equation*}
    \ell_1 \left( \dfrac{2h_M(\xi)-\langle u_1+u_2,\xi \rangle}{ \langle \ell_1 , \xi \rangle} \right)  = \ell_2 \left( \dfrac{2h_M(\xi)-\langle u_1+u_2,\xi \rangle}{ \langle  \ell_2 , \xi \rangle} \right).
\end{equation*}
This shows that either $\ell_1$ is parallel to $\ell_2$, or $\left\langle \frac{u_1+u_2}{2},\xi \right\rangle = h_M(\xi)$. Since $u_1$ and $u_2$ are fixed and $\xi \in S^1/(X \cup \{\ell_1^\perp,\ell_2^\perp,m_1^\perp,m_2^\perp\})$, almost every supporting line of $M$ must pass through a common point $\frac{u_1+u_2}{2}$. This is impossible for a non-trivial hedgehog $M$.

Since both cases yield a contradiction, there is no point at which $P$ and $Q$ intersect transversely. As they are not subsets of each other, they must coincide, $P= Q$.
\end{proof} 

\textbf{\subsection{Proof of Theorem \ref{Thm2}}}
\begin{proof}
    Suppose to the contrary that $P$ and $Q$ are origin symmetric distinct polygons in $\R^2$ with an analytic hedgehog $M$ fully contained in their interiors, $M \subset \text{int}(P \cap Q)$, yet
    \[ \vol{2}{P \cap S_M(\xi)} = \vol{2}{Q \cap S_M(\xi)}, \quad \forall \xi \in S^1.\]    
    Note that neither of the polytopes can be a subset of the other, otherwise $\vol{2}{P \cap S_M(\xi)} \neq \vol{2}{Q \cap S_M(\xi)}$. Since their interiors intersect (as they both contain $M$), there must exist a point $u \in \partial P \cap \partial Q$. Moreover, $u$ can be chosen so that the sides of the polygons intersect transversally at $u$. Consequently, the parametrizations of the lines containing the corresponding sides of $P$ and $Q$ containing $u$ are,
    \begin{equation}
        x(s_1) = u + s_1\ell_1, \qquad x(t_1) = u + t_1m_1 , \qquad \ell_1 , m_1 \in S^1, \quad s_1,t_1 \in \R.\label{eqn:edges3}
    \end{equation}
Since the intersection is transversal, we have $\ell_1$ is not parallel to $m_1$. 

For the standard basis $\{e_1,e_2\}$ of $\R^2$, consider a coordinate system such that $u$ belongs to the horizontal line $\langle x , e_2 \rangle = h_M(e_2)$ and the half plane $\langle x , e_1 \rangle < 0$. Since the direction vectors of the sides cannot be horizontal, let $\langle \ell_1 , e_2 \rangle > 0$ and $\langle m_1 , e_2 \rangle > 0$.
Let $v, w$ be the other points of intersection of the line $\langle x , e_2 \rangle = h_M(e_2)$ with $P$ and $Q$, respectively. The lines containing the points $v,w$ and the corresponding lines are,
    \begin{equation}
        x(s_2) = v + s_2\ell_2 , \qquad x(t_2) = w + t_2 m_2, \qquad \ell_2,m_2 \in S^1, \quad s_2,t_2 \in \R.\label{eqn:edges4}
    \end{equation} 
Similarly assume $\langle \ell_2,e_2\rangle < 0$ and $\langle m_2,e_2 \rangle < 0$. 

Let $\Omega \subset S^1$ be an open set of unit vectors $\xi$ such that the line $\langle x,\xi \rangle = h_M(\xi)$ intersects the sides given by equations \eqref{eqn:edges3} and \eqref{eqn:edges4} but contains no vertices. Moreover, if $\xi \in \Omega$, $\xi(\phi) = (\cos(\frac{\pi}{2}-\phi),\sin(\frac{\pi}{2}-\phi)) = (\sin\phi,\cos\phi)$. 
Then the line $\langle x,\xi\rangle =h_M(\xi)$ intersects $P$ and $Q$ at 4 points that are determined by the systems of equations,
\[\begin{cases}
    \langle x , \xi \rangle = h_M(\xi) \\
    x = u + \ell_1 s_1 
\end{cases} \quad  \quad \begin{cases}
    \langle x , \xi \rangle = h_M(\xi) \\
    x = u + m_1 t_1
\end{cases}   \quad
\begin{cases}
    \langle x , \xi \rangle = h_M(\xi) \\
    x = v + \ell_2 s_2 
\end{cases} \quad  \quad \begin{cases}
    \langle x , \xi \rangle = h_M(\xi) \\
    x = w + m_2 t_2.
\end{cases} \]  

This yields
\begin{align*}
    p_1 &= u +  \left( \frac{h_M(\xi) - \langle u ,\xi\rangle }{\langle \ell_1 ,\xi \rangle }\right)\ell_1, \qquad
    q_1 = u +  \left( \frac{h_M(\xi) - \langle u ,\xi\rangle }{\langle m_1 ,\xi \rangle }\right)m_1, \\
    p_2 & = v + \left( \frac{h_M(\xi) - \langle v ,\xi\rangle }{\langle \ell_2 ,\xi \rangle }\right)\ell_1, \qquad
    q_2 = w +  \left( \frac{h_M(\xi) - \langle w ,\xi\rangle }{\langle m_2 ,\xi \rangle }\right)m_2.
\end{align*}
Consider the following sets for $P$ and $Q$,
\begin{align*}
    X  = \vol{2}{P \cap S_M(\xi)\cap S_M(e_2)},\qquad  & X'  = \vol{2}{Q \cap S_M(\xi)\cap S_M(e_2)},\\
    Y = \vol{2}{P \cap S_M(\xi) / S_M(e_2)}, \qquad  & Y' = \vol{2}{Q \cap S_M(\xi) / S_M(e_2)}, \\
Z  = \vol{2}{P \cap S_M(e_2) / S_M(\xi)}, \qquad  & Z'  = \vol{2}{Q \cap S_M(e_2) / S_M(\xi)}.  
\end{align*}
By assumption of the theorem, 
$$\vol{2}{P \cap S_M(\xi)} = \vol{2}{Q \cap S_M(\xi)}, \quad \vol{2}{P \cap S_M(e_2}) = \vol{2}{Q \cap S_M(e_2)},$$ 
which implies that $X+Y = X'+Y'$ and $X+Z= X'+Z'$. Hence, $Y-Z = Y'-Z'$ that reads
\begin{equation*}
\begin{split}
        &\vol{2}{P \cap S_M(\xi) / S_M(e_2)} -  \vol{2}{P \cap S_M(e_2) / S_M(\xi)} \\
    &=  \vol{2}{Q \cap S_M(\xi) / S_M(e_2)} -  \vol{2}{Q \cap S_M(e_2) / S_M(\xi)}, \label{eqn:equality1}
\end{split}
\end{equation*}
or equivalently,
\begin{equation}
    \begin{split}
        &\vol{2}{P \cap S_M(\xi) / S_M(e_2)} - \vol{2}{Q \cap S_M(\xi) / S_M(e_2)}\\ 
    &= \vol{2}{P \cap S_M(e_2) / S_M(\xi)} - \vol{2}{Q \cap S_M(e_2) / S_M(\xi)}. \label{eqn:equality-tri-quad}
    \end{split}
\end{equation}
Therefore, for $\xi \in \Omega$, the area of the triangle with vertices of $u$, $p_1$ and $q_1$ is equal to the area of quadrilateral with vertices of $v$, $w$, $p_2$ and $q_2$.

Recall that the area $A$ of a simple polygon (a polygon without self intersections) with vertices $x_1,x_2,x_3,\ldots,x_k$ in counterclockwise order is known from \cite{Braden1986},
\[ A = \frac{1}{2} \left( |x_1,x_2| + |x_2,x_3| + \ldots + |x_{k-1} , x_k| + |x_k,x_1| \right),\]

where $| \cdot , \cdot|$ is the determinant of $2 \times 2$ matrix. 
If $u$, $q_1$ and $p_1$ are in counterclockwise orientation then $v$, $w$, $q_2$, $p_2$ is also in counterclockwise orientation (otherwise, the areas of the triangle and the quadrilateral signs do not coincide).

Denote the corresponding values of the parameters for the four points as $a= \frac{h_M(\xi) - \langle u ,\xi\rangle }{\langle m_1 ,\xi \rangle }$, $b =  \frac{h_M(\xi) - \langle u ,\xi\rangle }{\langle \ell_1 ,\xi \rangle }$, $c =  \frac{h_M(\xi) - \langle v ,\xi\rangle }{\langle \ell_2 ,\xi \rangle } $ and lastly $d =  \frac{h_M(\xi) - \langle w ,\xi\rangle }{\langle m_2 ,\xi \rangle }$. To determine the area $A_\triangle$ of the triangle with vertices $u, q_1,p_1$, we evaluate the following,
\begin{align*}
    |u,q_1| & = \left| u , u +b m_1  \right| = |u,u| + b|u,m_1|  = b|u,m_1|,\\
    |q_1,p_1| & = \left| u + bm_1 , u + a\ell_1  \right| = |u,u| +a|u,\ell_1| + b|m_1,u| + ab|m_1,\ell_1| \\
    & = a|u,\ell_1| - b|u,m_1| + ab|m_1,\ell_1|, \\
    |p_1,u| & = |u + a\ell_1 , u | = |u,u| - a |u,\ell_1| = -a|u,\ell_1 |.
\end{align*}
The sum yields 
    $A_{\triangle} = \frac{1}{2} ab|m_1,\ell_1| = \frac{1}{2}\left( \frac{(h_M(\xi) - \langle u, \xi \rangle)^2}{\langle \ell_1,\xi \rangle\langle m_1,\xi \rangle} \right)|m_1,\ell_1|$. 
For the area $A_\square$ of the quadrilateral with vertices $v,w,q_1,p_1$, the corresponding determinants are
\begin{align*}
    |v,w| & = -|w,v|, \\
    |w,q_2| &= |w , w + dm_2| = |w,w| + d|w,m_2|  = d|w,m_2|,\\ 
|q_2,p_2| & = |w + dm_2, v +c\ell_2| = |w,v| + c|w,\ell_2| - d |v,m_2|  + cd|m_2,\ell_2|, \\
|p_2,v| & = |v + c\ell_2 , v | = |v,v| + c | \ell_2 , v| = -c|v,\ell_2|.
\end{align*}
The sum yields the area $A_\square$,
\begin{equation*}
    \begin{split}
        A_{\square} &  = \frac{1}{2}(c |w-v,\ell_2| + d|w-v,m_2| +cd|m_2,\ell_2| )\\ 
        &= \left(\frac{h_M(\xi) - \langle v ,\xi\rangle }{2\langle \ell_2 ,\xi \rangle } \right)|w-v, \ell_2| +\left( \frac{h_M(\xi) - \langle w ,\xi\rangle }{2\langle m_2 ,\xi \rangle } \right) |w-v,m_2|\\
        & + \frac{(h_M(\xi) - \langle v ,\xi\rangle)(h_M(\xi) - \langle w ,\xi\rangle)}{2\langle \ell_2 , \xi \rangle \langle m_2, \xi \rangle}|m_2,\ell_2|.
    \end{split} \label{eqn:areaquad}
\end{equation*}

By $(\ref{eqn:equality-tri-quad})$, $A_\triangle = A_\square$ that reads 
\begin{equation}
    \begin{split}
         \left( \frac{(h_M(\xi) - \langle u, \xi \rangle)^2}{\langle \ell_1,\xi \rangle\langle m_1,\xi \rangle} \right)|m_1,\ell_1|  = &\left(\frac{h_M(\xi) - \langle v ,\xi\rangle }{\langle \ell_2 ,\xi \rangle } \right)|w-v, \ell_2| \\
         &+\left( \frac{h_M(\xi) - \langle w ,\xi\rangle }{\langle m_2 ,\xi \rangle } \right) |w-v,m_2|\\ &+
    \frac{(h_M(\xi) - \langle v ,\xi\rangle)(h_M(\xi) - \langle w ,\xi\rangle)}{\langle \ell_2 , \xi \rangle \langle m_2, \xi \rangle}|m_2,\ell_2|.
    \end{split} \label{eqn:bigone}
\end{equation}

Let $u = \alpha e_1 + h_M(e_2)e_2$, $v = \beta e_1 + h_M(e_2)e_2$, and $w = \gamma e_1 + h_M(e_2)e_2$, where $\alpha< 0$, $\beta > 0 , \gamma > 0$. Since $h(\phi) = h_M(\xi(\phi))$ is real analytic on $S^1$, then $h_M$ is equal to its power series around $e_2$ ($\phi=0$) on $\Omega$,
\[ 
h_M(\xi) = h_M(\sin\phi,\cos\phi)= h(0) + \frac{d h}{d \phi} (0)\phi + o(\phi^2).
\]
By the chain rule for $\xi \in \Omega$, $x_1 = \sin\phi$, $x_2 = \cos\phi $,
\begin{align*}
    \frac{d h}{d \phi} (0) = \frac{\partial h}{\partial x_1}\frac{\partial x_1}{\partial \phi} + \frac{\partial h}{\partial x_2}\frac{\partial x_2}{\partial \phi} = \left. \frac{\partial h}{\partial x_1} \cos\phi\right|_{\phi=0}  - \left. \frac{\partial h}{\partial x_2} \sin\phi\right|_{\phi=0}
 = \left.  \frac{\partial h}{\partial x_1}\right|_{\phi=0} = h_{x_1}(0).
\end{align*} 
Hence, the Taylor expansion of $h$ about $\phi = 0$ is
\begin{equation}
    h(\phi) = h(0) +h_{x_1}(0)\phi + o(\phi^2). \label{eqn:supportexp}
\end{equation} 
Also consider the Taylor expansion of $\langle x,\xi(\phi) \rangle$ for sufficiently small $\phi$, where $x=\lambda e_1 + h(e_2)e_2 $,
\begin{equation*}
\begin{split}
    \langle x , \xi(\phi) \rangle & = \lambda \sin\phi + h_M(e_2) \cos\phi   = \lambda(\phi + o(\phi^3) ) + h_M(e_2)(1 + o(\phi^2)) \\
    & = h_M(e_2) + \lambda \phi + o(\phi^2).
    \end{split} \label{eqn:innerexp}
\end{equation*}
Then
\begin{equation*}
    h_M(\xi(\phi)) - \langle x ,\xi(\phi)\rangle  = (h_{x_1}(0)-\lambda) \phi + o(\phi^2)= -\tilde{\lambda}\phi + o(\phi^2). \label{eqn:diffexp}
\end{equation*}
Since the above holds for all points $x$ on the horizontal line $\langle x ,e_2\rangle = h_M(e_2)$, for $x = u,v,w$, we have
 \begin{align*}
      h_M(\xi) - \langle u ,\xi\rangle  &= -\tilde{\alpha}\phi + o(\phi^2), \quad   h_M(\xi) - \langle v ,\xi\rangle = -\tilde{\beta}\phi + o(\phi^2),\\
      &h_M(\xi) - \langle w ,\xi\rangle  = -\tilde{\gamma}\phi + o(\phi^2) .
 \end{align*}
Furthermore, for $m_i = (m_{i_1},m_{i_2})$, $\ell_i = (\ell_{i_1},\ell_{i_2})$ the denominators in \eqref{eqn:bigone} can also be expanded around $e_2$, 
\begin{align*}
    \langle \ell_i,\xi \rangle & = \ell_{i_1} \sin\phi + \ell_{i_2}\cos\phi = \ell_{i_2} + \ell_{i_1}\phi + o(\phi^2),\\
    \langle m_i,\xi \rangle & = m_{i_1} \sin\phi + m_{i_2}\cos\phi = m_{i_2} + m_{i_1}\phi + o(\phi^2).
\end{align*}
For the reciprocal of $\langle \ell_i ,\xi \rangle $ and $\langle m_i,\xi \rangle$, we get
\begin{align*}
    \frac{1}{ \langle \ell_i,\xi \rangle} & = \frac{1}{\ell_{i_2} + \ell_{i_1}\phi + o(\phi^2)}
 = \dfrac{1}{\ell_{i_2}}\left( \frac{1}{1 + \epsilon(\phi)} \right)
     = \dfrac{1}{\ell_{i_2}}\left( 1 -\frac{\ell_{i_1}}{\ell_{i_2}} \phi + o(\phi^2)\right),
\end{align*}
where $\epsilon(\phi) = \frac{\ell_{i_1}}{\ell_{i_2}} \phi + o(\phi^2)$.
Thus,
\begin{align*}
     \frac{1}{ \langle \ell_i,\xi \rangle} = \frac{1}{\langle \ell_i ,e_2\rangle} \left( 1 - \frac{\langle \ell_i , e_1 \rangle}{\langle \ell_i,e_2,\rangle}\phi + o(\phi^2) \right), \,
    \frac{1}{ \langle m_i,\xi \rangle}  = \frac{1}{\langle m_i ,e_2\rangle} \left( 1 - \frac{\langle m_i , e_1 \rangle}{\langle m_i,e_2,\rangle} \phi+ o(\phi^2) \right).
\end{align*}

Consequently, for sufficiently small $\phi$, the expansions of the left-hand side (LHS) and the three terms of the right-hand side (RHS) in \eqref{eqn:bigone} are,
\begin{align*}
LHS &= \frac{\tilde{\alpha}^2}{\langle \ell_1,e_2 \rangle\langle m_1,e_2 \rangle} \phi^2 + o(\phi^3), \quad
RHS_{(1)} = \frac{\tilde{\beta}}{\langle\ell_2, e_2 \rangle}|v-w,\ell_2| \phi + o(\phi^2), \\
RHS_{(2)} & = \frac{\tilde{\gamma}}{\langle m_2, e_2 \rangle}|v-w,m_2|\phi+ o(\phi^2), \quad
RHS_{(3)}  = \frac{\tilde{\beta}\tilde{\gamma}}{\langle\ell_2, e_2 \rangle\langle m_2,e_2 \rangle}|m_2,\ell_2| \phi^2+ o(\phi^3).
\end{align*}

Equating the expansions of the two sides and collecting linear terms in $\phi$ yields
\begin{equation}
    \frac{\tilde{\beta}}{\langle\ell_2,e_2 \rangle} | v-w,\ell_2| + \frac{\tilde{\gamma}}{\langle m_2 , e_2 \rangle }|v-w, m_2 | = 0. \label{eqn:lin-terms}
\end{equation}
Note $v - w = (\beta - \gamma ) e_1$, then
$
|v-w, \ell_2 | = (\beta- \gamma) \langle \ell_2 ,e_2\rangle.
$
Analogously for $m_2$, $|v-w , m_2| = (\beta - \gamma) \langle m_2,e_2\rangle$. Then $(\ref{eqn:lin-terms})$ gives
$
(\beta - \gamma) ( \tilde{\beta} + \tilde{\gamma}) = 0.
$

If $\beta = \gamma$, then $v = w$.  Otherwise, $ \tilde{\beta} = - \tilde{\gamma}$, then by the definition of $\tilde{\beta}$ and $\tilde{\gamma}$, this yields
$
h_{x_1}(e_2) = \frac{\beta  +  \gamma}{2} .\label{eqn:contradiction2}
$ Furthermore, note that the midpoint of $v$ and $w$ is then given by $\left(\frac{\beta+\gamma}{2},h_M(e_2) \right) = \left( h_{x_1}(e_2),h(e_2) \right)$.

By \eqref{def:supporting_point}, a support point $X_\xi \in M$ for $\xi = (\sin\phi,\cos\phi)$ is
$X_\xi = h_M(\xi) \xi + \frac{dh_M(\xi)}{d\phi} \xi^\perp$. Note that $\xi^\perp = (\cos\phi,-\sin\phi)$.
If $\xi = e_2$ (that is $\phi =0$), then $\xi ^\perp = e_1$. Using the chain rule, $\frac{d h}{d \phi} (e_2) = h_{x_1} (e_2)$. This yields that
$X_{e_2} = h(e_2) e_2  +h_{x_1}(e_2)e_1 = (h_{x_1}(e_2),h(e_2))$.
Thus, the support point $X_{e_2}$ of $M$ is the midpoint between $v$ and $w$. In particular, that means $X_{e_2} \in M$, $X_{e_2} \not \in P \cap Q$, which is impossible since $M \subset P \cap Q$, so $\tilde{\beta} \neq -\tilde{\gamma}$.

Under the assumption $\beta = \gamma$ (or equivalently $v=w$), $(\ref{eqn:bigone})$ results in
 $
      \frac{(h_M(\xi) - \langle u, \xi \rangle)^2}{\langle \ell_1,\xi \rangle\langle m_1,\xi \rangle} |m_1,\ell_1| =  \frac{(h_M(\xi) - \langle v, \xi \rangle)^2}{\langle \ell_2,\xi \rangle\langle m_2,\xi \rangle} |m_2,\ell_2| .\label{eqn:equality-v-w}
 $
By Theorem \ref{thrm:analytic-combination} both sides are analytic, and by Theorem $\ref{thrm:eqnext}$ their analytic extensions are unique and coincide beyond the intersected sides of the polygons. Then, the points of non-analyticity of the left and right hand side must coincide. Thus, $\{\ell_1^\perp, m_1^\perp\} = \{\ell_2^\perp, m_2^\perp\}$.
Furthermore, it implies that the pair of vectors $\{\ell_1, m_1\}$ must have the same orientation as $\{\ell_2, m_2\}$ (otherwise, the signs of the left side and the right side are distinct).
Then the only possible outcome is $\ell_1 =-\ell_2$ and $m_1 = -m_2$. To conclude the proof, we show that this is impossible.

Since $P$ and $Q$ are symmetric in origin,  $-u \in P \cap Q$ if $u \in P \cap Q$. Then the sides of $P$ and $Q$ containing $-u$ have directions $\ell_1, m_1$, respectively. Hence, $u,v,-u$ belong to the pairwise parallel sides of $P$ and $Q$. As a convex polygon cannot have three distinct parallel sides, then $u$ and $-v$ belong to the same sides of $P$ and $Q$. Therefore, the side of $P$ that is adjacent to $-u$ and parallel to $\ell_1$ should be the \textit{same} as the side of $P$ that is adjacent to $v$ and parallel to $\ell_1$. Thus, $-u$ lies on the straight line that passes through $v$ and is parallel to $\ell_1$. Similarly, $-u$ lies on the straight line that passes through $v$ and is parallel to $m_1$. Since $\ell_1$ is not parallel to $m_1$, this gives $v = -u$. This is not possible and implies $P=Q$.
\end{proof}
Note that the above proof relies on the fact that a supporting line $H(\xi)$ that defines a slab and contains a common point of $P$ and $Q$ corresponds to an analytical point $\xi$ of $h_M$. 
\begin{corollary}
    Theorem \ref{Thm2} still holds if $h_M(\xi)$ is piecewise analytic on $S^1$ and analytic for $\xi$ such that $H(\xi)$ contains a common point of $P$ and $Q$.
\end{corollary}

\section*{Acknowledgements}

I would like to express gratitude to my supervisor, Dr.~Serhii Myroshnychenko, for the guidance and support provided throughout this research project. I also thank the Department of Mathematics and Statistics at UFV for the many opportunities and encouragement offered during my studies. In addition, I would like to acknowledge NSERC for the Undergraduate Student Research Award, which provided financial support and allowed me to pursue research in a subject I am deeply passionate about. Finally, I thank anonymous reviewers for their valuable comments and suggestions, which helped improve this work.

\footnotesize
\bibliographystyle{plain}
\bibliography{Paper_Biblography}
\normalsize

\end{document}